\documentclass[11pt,reqno]{amsart}

\usepackage{amsmath,amssymb,amsthm}
\usepackage{enumitem}
\usepackage{tikz}
\usetikzlibrary{arrows.meta}
\usepackage[margin=1in]{geometry}
\usepackage[hidelinks]{hyperref}

\numberwithin{equation}{section}

\theoremstyle{plain}
\newtheorem{theorem}{Theorem}[section]
\newtheorem{lemma}[theorem]{Lemma}
\newtheorem{proposition}[theorem]{Proposition}
\newtheorem{corollary}[theorem]{Corollary}
\newtheorem{conjecture}[theorem]{Conjecture}

\theoremstyle{definition}
\newtheorem{definition}[theorem]{Definition}
\newtheorem{construction}[theorem]{Construction}

\newcommand{\Kk}{K_k}
\newcommand{\mK}{m_K}
\newcommand{\mA}{m_A}
\newcommand{\mCr}{m_{\mathrm{cr}}}
\newcommand{\sA}{s_A}
\newcommand{\sK}{s_K}
\newcommand{\Kp}{K'}

\title[Antimagic orientations of graphs with a dominating clique]{Antimagic orientations of
graphs with a dominating clique}
\author{Abdulrahman Al-Taweel}
\address{Carnegie Mellon University Qatar, Doha, Qatar}
\date{}

\begin{document}

\begin{abstract}
A graph $G$ with $m$ edges has an \emph{antimagic orientation} if its edges can be oriented and
bijectively labelled $1,\dots,m$ so that the oriented vertex sums, the labels on in-arcs minus the
labels on out-arcs, are pairwise distinct. Hefetz, M\"utze and Schwartz conjectured that every connected graph
admits an antimagic orientation. We prove that this conjecture holds for every graph with a
dominating clique.
\end{abstract}

\maketitle
\pagestyle{plain}

\section{Introduction}\label{sec:intro}

An \emph{antimagic labelling} of a graph $G$ with $m$ edges is a bijection from $E(G)$ to
$\{1,\dots,m\}$ such that the vertex sums, each of which is the sum of the labels on its incident
edges, are pairwise distinct; $G$ is \emph{antimagic} if it admits one. Hartsfield and Ringel
\cite{HartsfieldRingel1990} introduced antimagic labellings in 1990 and made the following
conjecture.

\begin{conjecture}\label{conj:hr}
Every connected graph other than $K_2$ is antimagic.
\end{conjecture}

\noindent This conjecture is open in general, though it is proven to hold for many structured
families; see the dynamic survey of Gallian \cite{GallianDS6}.

Hefetz, M\"utze and Schwartz \cite{HefetzMutzeSchwartz2010} introduced the directed analogue. An
\emph{antimagic orientation} of $G$ is an orientation $D$ of $G$ together with a bijection
$\tau\colon A(D)\to\{1,\dots,m\}$ such that the \emph{oriented vertex sums}
\begin{equation}\label{eq:vsum}
s(v)\;=\;\sum_{\text{arcs }(u,v)\in D}\tau(uv)\;-\;\sum_{\text{arcs }(v,w)\in D}\tau(vw)
\end{equation}
are pairwise distinct. Motivated by Conjecture~\ref{conj:hr}, they proposed the following
\cite[Conjecture 4.2]{HefetzMutzeSchwartz2010}.

\begin{conjecture}\label{conj:hms}
Every connected graph admits an antimagic orientation.
\end{conjecture}

\noindent This conjecture has been verified for several classes, among them even regular graphs
\cite{LiSongWangYangZhang2017}, biregular bipartite graphs \cite{ShanYu2017Biregular}, and graphs of
large maximum degree \cite{YangCarlsonOwensPerrySinggihSongZhangZhang2019}, but remains open in
general.

A \emph{dominating clique} of $G$ is a clique $K$ such that every vertex of $G$ lies in $K$ or has a
neighbour in $K$. The antimagicness of graphs with a dominating clique belongs to a longer line of
work, all in the unoriented setting, on graphs built around a dominating substructure. Barrus
\cite{Barrus2010Canonical} used the canonical decomposition of graphs to establish antimagicness for
connected split graphs (a clique together with an independent set); Sl\'iva \cite{Sliva2012Regular} proved
antimagicness for a class of graphs with a regular dominating subgraph; and Yilma
\cite{Yilma2013LargeMaxDegree}, in a study of graphs with large maximum degree, established
antimagicness for a family of graphs with a dominating $K_2$. Specialising the dominating structure to a clique,
Beaudoire, Bentz and Picouleau \cite{BeaudoireBentzPicouleau2025DomClique} proved the following.

\begin{theorem}[{\cite[Theorem~4]{BeaudoireBentzPicouleau2025DomClique}}]\label{thm:bbp}
Let $G$ have a dominating clique $K$ of order at least four such that every vertex $v\notin K$
satisfies $\deg(v)\le\min_{u\in K}\deg(u)$. Then $G$ is antimagic.
\end{theorem}

\noindent Our main result is the following.

\begin{theorem}\label{thm:full}
Every graph $G$ with a dominating clique admits an antimagic orientation.
\end{theorem}

\noindent Theorem~\ref{thm:full} drops both the order bound and the degree hypothesis of
Theorem~\ref{thm:bbp}, and settles Conjecture~\ref{conj:hms} for graphs with a dominating clique $K$, with
no condition on the outside set $A=V(G)\setminus K$; in particular $G[A]$ may be disconnected,
edgeless, or empty.

\section{Preliminaries}\label{sec:defs}

We fix here the notation and the three edge classes used throughout. All
graphs are finite and simple. For $S\subseteq V(G)$, $G[S]$ is the induced subgraph; $\deg(v)$ is
the degree of $v$. A closed walk that uses every edge exactly once is an \emph{Euler tour}; an open
one is an \emph{Euler trail}. Recall that a connected multigraph has an Euler tour if and only if every vertex
has even degree, and an Euler trail if and only if exactly two vertices have odd degree. For an
orientation $D$ of $G$ and a bijection $\tau\colon A(D)\to\{1,\dots,m\}$, the oriented vertex sum
$s(v)$ is given by \eqref{eq:vsum}; $(D,\tau)$ is an antimagic orientation when the values $s(v)$ are
pairwise distinct.

Fix a dominating clique $K$ of order $k$ and set $A=V(G)\setminus K$. We first group $E(G)$ by endpoint location:
\begin{itemize}[topsep=2pt,itemsep=1pt]
\item the \emph{clique edges} $E(K)$, of which there are $\mK=\binom{k}{2}=\tfrac{k(k-1)}{2}$;
\item the \emph{inside-$A$ edges} $E(G[A])$, of which there are $\mA$;
\item the \emph{cross edges} between $K$ and $A$, of which there are $\mCr$.
\end{itemize}
Thus $m=\mK+\mA+\mCr$. Since $K$ dominates, each $v\in A$ has at least one cross edge; designate one
as the \emph{anchor} of $v$ and the others as its \emph{surplus}, so the cross-edge class consists of
$\mCr-|A|$ \emph{surplus cross edges} and $|A|$ \emph{anchor cross edges}.

These four edge sets receive the labels $\{1,\dots,m\}$ in consecutive blocks, from smallest
to largest:
\begin{equation}\label{eq:blocks}
\begin{aligned}
\text{clique edges} &:\quad \{1,\dots,\mK\},\\
\text{inside-$A$ edges} &:\quad \{\mK+1,\dots,\mK+\mA\},\\
\text{surplus cross edges} &:\quad \{\mK+\mA+1,\dots,m-|A|\},\\
\text{anchor cross edges} &:\quad \{m-|A|+1,\dots,m\}.
\end{aligned}
\end{equation}

\begin{definition}\label{def:sigma}
For $u\in K$, let $\sK(u)$ be the \emph{clique-internal sum} of $u$, namely the oriented sum of $u$
computed from the clique edges alone, or equivalently the contribution of $E(K)$ to $s(u)$. Let
$B_u$, the \emph{load} of $u$, be the sum of the cross labels on the arcs directed out of $u$ into
$A$. For $v\in A$, let $\sA(v)$ be the \emph{inside sum} of $v$, namely the oriented partial sum of
$v$ over the inside-$A$ edges alone.
\end{definition}

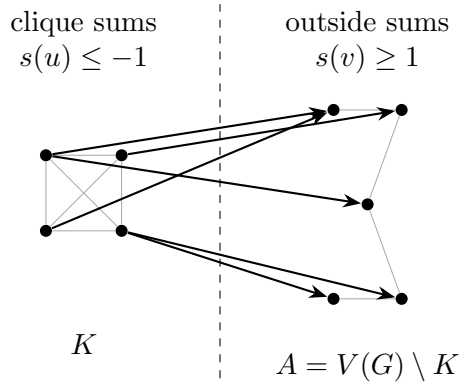
\begin{figure}[b]
\centering
\begin{tikzpicture}[
  x=1cm,y=1cm,
  v/.style={circle,fill,inner sep=1.6pt},
  ed/.style={gray!65},
  cross/.style={-{Stealth[length=2.4mm]},thick}]
  \node[v] (k0) at (-0.5,1.4){}; \node[v] (k1) at (0.5,1.4){};
  \node[v] (k2) at (-0.5,0.4){}; \node[v] (k3) at (0.5,0.4){};
  \foreach \i/\j in {k0/k1,k0/k2,k0/k3,k1/k2,k1/k3,k2/k3} \draw[ed] (\i)--(\j);
  \node[v] (a0) at (3.3,2.0){}; \node[v] (a1) at (4.2,2.0){};
  \node[v] (a2) at (3.3,-0.5){}; \node[v] (a3) at (4.2,-0.5){};
  \node[v] (a4) at (3.75,0.75){};
  \draw[ed] (a0)--(a1);
  \draw[ed] (a2)--(a3);
  \draw[ed] (a1)--(a4);
  \draw[ed] (a3)--(a4);
  \draw[cross] (k0)--(a0); \draw[cross] (k2)--(a0); \draw[cross] (k1)--(a1);
  \draw[cross] (k3)--(a2); \draw[cross] (k3)--(a3); \draw[cross] (k0)--(a4);
  \node at (0,-1.1){$K$}; \node at (3.75,-1.4){$A=V(G)\setminus K$};
  \node[align=center] at (0,2.9){clique sums\\$s(u)\le-1$};
  \node[align=center] at (3.75,2.9){outside sums\\$s(v)\ge 1$};
  \draw[dashed] (1.8,-1.55)--(1.8,3.4);
\end{tikzpicture}
\caption{Schematic of the sign separation obtained by the construction.}
\label{fig:sign}
\end{figure}

Throughout, every cross edge is oriented from its clique endpoint into its $A$-endpoint. With this
orientation, the label of each cross edge is added to the sum at its $A$-endpoint and
subtracted from the sum at its clique endpoint, so
\begin{equation}\label{eq:su-sv}
s(u)=\sK(u)-B_u\quad(u\in K),
\qquad
s(v)=\sA(v)+\!\!\sum_{\text{cross }e\text{ into }v}\!\!\tau(e)\quad(v\in A).
\end{equation}
In the labelled orientation constructed below, every clique vertex has a negative sum and every
outside vertex a positive sum, so the two groups are separated by sign (Figure~\ref{fig:sign}). This
sign separation, with the sorted labelling of the two sides, adapts to oriented sums the undirected
argument of Beaudoire, Bentz and Picouleau \cite{BeaudoireBentzPicouleau2025DomClique}; the lower
bound on the inside sums and the Euler-tour orientations underlying the clique interior are drawn
from Yang et al.\ \cite{YangCarlsonOwensPerrySinggihSongZhangZhang2019}.

Corollary~\ref{cor:floor} gives a lower bound on the inside sums $\sA(v)$, and
Lemma~\ref{lem:clique} gives pairwise distinct clique-internal sums $\sK(u)$ with a uniform upper
bound.

\section{A lower bound via Euler tours}\label{sec:floor}

\begin{lemma}[{Yang et al.\ \cite[Lemma 2.1]{YangCarlsonOwensPerrySinggihSongZhangZhang2019}}]\label{lem:floor}
Let $H$ be a graph with $m\ge1$ edges, and let $a_1<a_2<\dots<a_m$ be positive integers. Then there
exist an orientation $D$ of $H$ and a bijection $\tau\colon A(D)\to\{a_1,\dots,a_m\}$ with
$s(v)\ge-a_m$ for every $v\in V(H)$.
\end{lemma}

\begin{corollary}\label{cor:floor}
For an arbitrary $G[A]$, there exist an orientation of the inside-$A$ edges and a labelling of them by
$\{\mK+1,\dots,\mK+\mA\}$ such that $\sA(v)\ge-(\mK+\mA)$ for every $v\in A$. If $\mA=0$, then
$\sA(v)=0$ for every $v\in A$.
\end{corollary}

\begin{proof}
Apply Lemma~\ref{lem:floor} to $H=G[A]$ with $\{a_1,\dots,a_{\mA}\}=\{\mK+1,\dots,\mK+\mA\}$, whose
largest element is $\mK+\mA$ when $\mA\ge1$; the case $\mA=0$ is the empty labelling.
\end{proof}

\section{The clique interior}\label{sec:clique}

We orient and label the clique edges by the smallest block $\{1,\dots,\mK\}$, identifying the
dominating clique $K$ with $\Kk$ on vertex set $\{v_0,v_1,\dots,v_{k-1}\}$.

\begin{lemma}\label{lem:clique}
For every integer $k\ge1$ there exist an orientation $D$ of $\Kk$ and a bijection
$\tau\colon A(D)\to\{1,\dots,\mK\}$ such that the $k$ oriented sums $\sK(u)$ are pairwise distinct
and
\[
\sK(u)\;\le\;\mK\qquad\text{for every }u.
\]
Moreover, the multiset $\{\sK(u)\}$ may be assigned to the vertices of $\Kk$ in any prescribed order.
\end{lemma}

The constructions below use the following Euler-tour calculation. On a connected even-degree graph
with $m'$ edges, orient the edges along a closed Euler tour and label them consecutively
$L+1,L+2,\dots,L+m'$; then each non-start vertex $v$ has oriented sum $-\tfrac{\deg(v)}{2}$ and the
start vertex $z$, the \emph{seam}, has $m'-\tfrac{\deg(z)}{2}$, independently of $L$, since each
non-wrap-around passage pairs an incoming edge labelled $\ell$ with the following outgoing edge
labelled $\ell+1$ and contributes $-1$ while the wrap-around at $z$ contributes $m'-1$.

We prove Lemma~\ref{lem:clique} in three ranges: odd $k\ge5$, even $k\ge4$, and $k\le3$.

\subsection{Odd \texorpdfstring{$k\ge5$}{k>=5}}\label{sub:odd}

Write $k=2d+1$ with $d=\tfrac{k-1}{2}\ge2$, and $\mK=kd$. Each vertex of $\Kk$ has even degree
$k-1=2d$, so $\Kk$ is Eulerian.

\begin{construction}[odd $k$]\label{con:odd}
Take vertices $\{v_0,v_1,\dots,v_{k-1}\}$ and set the seam $z:=v_{k-1}$.
\begin{enumerate}[label=\textbf{Step \arabic*.},leftmargin=*,topsep=2pt,itemsep=2pt]
\item Let $P=(z,v_0,v_1,\dots,v_{k-2})$ be the Hamilton path with edges $zv_0$, $v_0v_1$,
$v_1v_2,\dots,v_{k-3}v_{k-2}$, which has $k-1=2d$ edges. In $\Kk-E(P)$ every vertex has even degree
except the two ends $v_{k-2}$ and $z$, and the graph is connected (after removing a Hamilton path from
a clique on $k\ge5$ vertices, any two vertices are adjacent or have a common neighbour). Hence
$\Kk-E(P)$ has an Euler trail $Q$
from $v_{k-2}$ to $z$. Concatenating $P$ and $Q$ gives a closed Euler tour $\mathcal T$ of $\Kk$ at $z$.
\item Label the edges $1,\dots,\mK$ along $\mathcal T$ and orient each
along the traversal. The first $2d$ labels fall on $P$; the alternate edges $v_{2i}v_{2i+1}$ of $P$
receive labels $2,4,\dots,2d$, since edge $v_{2i}v_{2i+1}$ is the $(2i+2)$-th edge of $\mathcal T$. The
Euler-tour calculation above (with $L=0$ and degree $2d$) gives the tour sums
$\tilde{s}(z)=\mK-d$ and $\tilde{s}(v)=-d$ for $v\ne z$.
\item Let $M=\{v_{2i}v_{2i+1}:i=0,\dots,d-1\}$ be the
alternate-edge matching of $P$; it covers $v_0,\dots,v_{k-2}$ and misses $z$, and its labels are
$\{2,4,\dots,2d\}$, each oriented $v_{2i}\to v_{2i+1}$. Reverse every edge of $M$; the resulting
orientation and labelling form the final $(D,\tau)$.
\end{enumerate}
\end{construction}

\begin{proof}[Proof of Lemma~\ref{lem:clique}, odd $k\ge5$]
Let $a_i:=2(i+1)$ be the label of the reversed edge $v_{2i}v_{2i+1}$. Reversing an arc labelled $a$
raises its tail's sum by $2a$ and lowers its head's sum by $2a$; hence $v_{2i}$ gains $+2a_i$ and
$v_{2i+1}$ loses $2a_i$. The seam $z$ is not incident with an edge of $M$, so its sum is unchanged. Applying these changes to the
Step~2 tour sums $\tilde{s}$ gives
\begin{equation}\label{eq:odd-sums}
\sK(z)=\mK-d,\qquad \sK(v_{2i})=-d+4(i+1),\qquad \sK(v_{2i+1})=-d-4(i+1)\quad(i=0,\dots,d-1).
\end{equation}

\smallskip
\noindent\emph{Distinctness.} Put $S^{+}=\{-d+4(i+1):0\le i\le d-1\}$ and $S^{-}=\{-d-4(i+1):0\le i\le d-1\}$.
Within $S^{+}$ the values strictly increase in $i$; within $S^{-}$ they strictly decrease; so each
block has $d$ distinct values. Every element of $S^{+}$ exceeds $-d$ and every element of $S^{-}$ is
below $-d$, so they are disjoint and neither contains $-d$; the $2d$ matched sums are pairwise
distinct. Finally $\sK(z)=2d^2$ while $\max S^{+}=3d$, and $2d^2>3d$ for $d\ge2$, so $\sK(z)$ is
distinct from all.

\smallskip
\noindent\emph{Bound.} By \eqref{eq:odd-sums}, $\sK(z)=\mK-d$ and $\sK(u)\le3d$ for $u\ne z$; both
are at most $\mK$, so $\sK(u)\le\mK$.
\end{proof}

\subsection{Even \texorpdfstring{$k\ge4$}{k>=4}}\label{sub:even}

Write $k=2d+2$ with $d=\tfrac{k-2}{2}\ge1$, and $\mK=(d+1)(2d+1)$. Each vertex of $\Kk$ has odd
degree $k-1=2d+1$, so $\Kk$ is not Eulerian; removing a
perfect matching makes every degree even.

\begin{construction}[even $k$]\label{con:even}
Take vertices $\{v_0,\dots,v_{2d+1}\}$ and set the seam $z:=v_{2d+1}$.
\begin{enumerate}[label=\textbf{Step \arabic*.},leftmargin=*,topsep=2pt,itemsep=2pt]
\item Let $M=\{v_{2i}v_{2i+1}:i=0,\dots,d\}$ be the perfect matching pairing $v_0$--$v_1,\dots,v_{2d}$--$v_{2d+1}$;
it has $d+1$ edges and covers all vertices. Then $\Kk-M$ is $2d$-regular and connected (a matched pair shares all $2d\ge2$ other vertices as common neighbours), hence Eulerian,
with $m'=\mK-(d+1)=2d(d+1)$ edges.
\item Take a closed Euler tour of $\Kk-M$ at $z$, orient along it,
and label the edges of $\Kk-M$ with $\{d+2,\dots,\mK\}$ in traversal order. The Euler-tour
calculation above (with $L=d+1$ and degree $2d$) gives the tour sums
$\tilde{s}(z)=\mK-2d-1$ and $\tilde{s}(v)=-d$ for $v\ne z$.
\item Orient each matching edge $v_{2i}\to v_{2i+1}$ and label
it $i+1$, so $M$ receives $\{1,\dots,d+1\}$.
\end{enumerate}
\end{construction}

\begin{proof}[Proof of Lemma~\ref{lem:clique}, even $k\ge4$]
The matching $M$ receives $\{1,\dots,d+1\}$ and the edges of $\Kk-M$ receive $\{d+2,\dots,\mK\}$;
together these assignments define a bijection onto $\{1,\dots,\mK\}$. Edge $v_{2i}v_{2i+1}$ contributes $-(i+1)$ at $v_{2i}$ and $+(i+1)$ at $v_{2i+1}$;
the seam $z=v_{2d+1}$ is the head of $v_{2d}v_{2d+1}$, gaining $+(d+1)$. Adding these matching
contributions to the tour sums $\tilde{s}$ gives
\begin{equation}\label{eq:even-sums}
\begin{aligned}
&\sK(z)=\mK-d,\qquad \sK(v_{2d})=-(2d+1),\\
&\sK(v_{2i+1})=-d+(i+1),\qquad \sK(v_{2i})=-d-(i+1)\quad(i=0,\dots,d-1),
\end{aligned}
\end{equation}
the seam value being $(\mK-2d-1)+(d+1)=\mK-d$.

\smallskip
\noindent\emph{Distinctness.} By \eqref{eq:even-sums}, the sums $\sK(v_{2i+1})=-d+(i+1)$ strictly
increase in $i$ and lie in $(-d,0]$, while the sums $\sK(v_{2i})=-d-(i+1)$ strictly decrease and lie
in $[-2d,-d)$; no vertex attains $-d$, so these two groups are disjoint. Moreover
$\sK(v_{2d})=-(2d+1)<-2d$ and $\sK(z)=2d^2+2d+1>0$, so all $k$ sums are distinct.

\smallskip
\noindent\emph{Bound.} By \eqref{eq:even-sums}, $\sK(z)=\mK-d$ and every other $\sK(u)\le0$, so
$\sK(u)\le\mK$.
\end{proof}

\subsection{Small \texorpdfstring{$k\le3$}{k<=3}}\label{sub:small}

The cases $k\le3$ lie outside the two constructions above and are checked
directly.

\begin{proposition}\label{prop:small}
Lemma~\ref{lem:clique} holds for $k\in\{1,2,3\}$.
\end{proposition}

\begin{proof}
By Definition~\ref{def:sigma}, an arc $x\to y$ labelled $a$ adds $-a$ to $\sK(x)$ and $+a$ to
$\sK(y)$. Take the following orientations and labels:
\begin{enumerate}[label=\textup{(\roman*)},topsep=2pt,itemsep=2pt]
\item $k=1$: $\mK=0$, and the single sum is the empty sum $\sK(v_0)=0=\mK$.
\item $k=2$: $\mK=1$; orient the edge $v_0\to v_1$ with label $1$, giving $\{\sK(u)\}=\{-1,1\}$.
\item $k=3$: $\mK=3$; take the transitive tournament with arcs $v_0\to v_1$ (label $3$), $v_0\to v_2$
(label $1$), $v_1\to v_2$ (label $2$), giving $\sK(v_0)=-4$, $\sK(v_1)=1$, $\sK(v_2)=3$, so
$\{\sK(u)\}=\{-4,1,3\}$.
\end{enumerate}
In each case the $k$ sums are pairwise distinct with $\max_u\sK(u)=\mK$.
\end{proof}

Together, Proposition~\ref{prop:small} and Constructions~\ref{con:odd}--\ref{con:even} prove
Lemma~\ref{lem:clique} for all $k\ge1$; the resulting sums may be assigned to the vertices of $\Kk$
in any prescribed order, since every permutation of $V(\Kk)$ is an automorphism of $\Kk$.

\section{The construction and the full theorem}\label{sec:engine}

The construction places no restriction on $G[A]$, which may be disconnected or edgeless. It requires
only that every vertex of $K$ have at least one neighbour in $A$. Each
outside vertex may carry several cross edges, with no bound on their number.

\begin{theorem}\label{thm:engine}
Let $G$ have a dominating clique $K$ of order $k\ge1$ such that every vertex of $K$ has a neighbour
in $A=V(G)\setminus K$. Then
$G$ admits an antimagic orientation.
\end{theorem}

\begin{construction}\label{con:engine}
\leavevmode
\begin{enumerate}[label=\textbf{Step \arabic*.},leftmargin=*,topsep=2pt,itemsep=2pt]
\item Orient and label $E(G[A])$ with $\{\mK+1,\dots,\mK+\mA\}$ by
Corollary~\ref{cor:floor}, so $\sA(v)\ge-(\mK+\mA)$ for every $v\in A$; if $\mA=0$, every
$\sA(v)=0$.
\item Orient every cross edge from $K$ into $A$. Give the surplus edges the
lower cross labels in any order and reserve the top $|A|$ cross labels for the anchors. Write
$\mathrm{sp}(v)$ for the sum of the surplus labels at $v$ and set $p(v)=\sA(v)+\mathrm{sp}(v)$. Order
$A$ as $p(v_1)\le\dots\le p(v_{|A|})$ and give $v_j$'s anchor the label $m-|A|+j$. This fixes each
load $B_u$.
\item Sort the clique vertices by load, $B_{u_1}\ge\dots\ge B_{u_k}$. Orient
and label $E(K)$ with $\{1,\dots,\mK\}$ by Lemma~\ref{lem:clique}, so the $k$ sums $\sK(u)$ are
pairwise distinct with $\sK(u)\le\mK$; using the freedom to assign $\{\sK(u)\}$ to the vertices in any
prescribed order, place them so that $\sK(u_1)<\dots<\sK(u_k)$, opposite to the load order.
\end{enumerate}
\end{construction}

\begin{lemma}\label{lem:surplus}
In Construction~\ref{con:engine}, $\mathrm{sp}(v)\ge0$ for every $v\in A$, and the outside sums
satisfy
\[
1\le s(v_1)<s(v_2)<\dots<s(v_{|A|}).
\]
\end{lemma}

\begin{proof}
The surplus labels are positive, so $\mathrm{sp}(v)\ge0$ for every $v$. By \eqref{eq:su-sv},
$s(v_j)=p(v_j)+(m-|A|+j)$, so for $i<j$
\[
s(v_j)-s(v_i)=\bigl(p(v_j)-p(v_i)\bigr)+(j-i)\ge0+1>0,
\]
the first bracket nonnegative because $A$ is ordered so that $p(v_1)\le\dots\le p(v_{|A|})$. For the
lower bound, $p(v_j)=\sA(v_j)+\mathrm{sp}(v_j)\ge-(\mK+\mA)$ by Corollary~\ref{cor:floor}, while
$\mCr\ge|A|$ gives $m-|A|=\mK+\mA+(\mCr-|A|)\ge\mK+\mA$, so the anchor label of $v_j$ satisfies
$m-|A|+j\ge\mK+\mA+1$. Hence $s(v_1)\ge-(\mK+\mA)+(\mK+\mA+1)=1$.
\end{proof}

\begin{proof}[Proof of Theorem~\ref{thm:engine}]
Since every vertex of $K$ has a neighbour in $A$, every $u\in K$ is incident with a cross edge. The blocks \eqref{eq:blocks} partition $\{1,\dots,m\}$ and each class is labelled by
its block, so $\tau$ is a bijection and $D$ orients every edge once. By Lemma~\ref{lem:surplus} the
outside sums are pairwise distinct and at least $1$. Two statements remain.

\smallskip
\noindent\emph{(a) Clique sums are pairwise distinct.} By Step~3, $s(u_r)=\sK(u_r)-B_{u_r}$
with $B_{u_r}$ nonincreasing and $\sK(u_r)$ strictly increasing, so $s(u_r)$ strictly increases.
Only distinctness of the clique sums and the opposite pairing are used; loads need not be distinct.

\smallskip
\noindent\emph{(b) Every clique sum is below every outside sum.} Since each $u\in K$ has a cross edge, its load contains a cross
label, and by \eqref{eq:blocks} every cross label is at least $\mK+\mA+1$, so $B_u\ge\mK+\mA+1$; with $\sK(u)\le\mK$,
\[
s(u)=\sK(u)-B_u\le\mK-(\mK+\mA+1)=-\mA-1\le-1<0<1\le s(v).
\]
Statement (b) puts the clique sums strictly below $1$ while the outside sums are at least $1$, so
with (a) all $k+|A|$ sums are distinct, and $(D,\tau)$ is antimagic.
\end{proof}

\subsection{Proof of Theorem~\ref{thm:full}}\label{sec:reduction}

To remove the neighbour hypothesis on $K$ it suffices to choose the dominating clique as small as possible.

\begin{proof}[Proof of Theorem~\ref{thm:full}]
If $G$ has no edge then $G$ is a single vertex, since in an edgeless graph a clique has at most one
vertex and can dominate no other vertex; the empty orientation of a single vertex is trivially
antimagic. Assume henceforth that $G$ has an edge. Among all dominating cliques of $G$
choose one, $K$, of minimum order $k\ge1$, and set $A=V(G)\setminus K$; we show that every vertex of
$K$ has a neighbour in $A$. If $k=1$, write $K=\{u\}$; since $G$ has an edge, it has a vertex $w\ne u$, and $w$ is
adjacent to $u$ because $K$ dominates, so $u$ has the neighbour $w\in A$. If $k\ge2$, suppose some
$u\in K$ had no neighbour in $A$, so that all neighbours of $u$ lie in $K$. Then $\Kp:=K\setminus\{u\}$
is a clique of order $k-1\ge1$ that still dominates $G$: every $x\in A$ has a neighbour in $K$, and
that neighbour is not $u$ (as $u$ has no neighbour in $A$), so it lies in $\Kp$, while $u$ is adjacent
to every vertex of $\Kp$. This contradicts the minimality of $K$. Hence every vertex of $K$ has a
neighbour in $A$, and Theorem~\ref{thm:engine} yields an antimagic orientation of $G$.
\end{proof}

\section*{Acknowledgement}

The author would like to thank Zelealem B. Yilma for suggesting the problem and for his guidance and
many helpful comments throughout this work.

\bibliographystyle{plain}
\bibliography{references}

\end{document}